\documentclass[11 pt]{amsart}
\usepackage{amssymb,amsmath,epsfig,mathrsfs, enumerate}
\usepackage{graphicx}
\usepackage[normalem]{ulem}
\usepackage{fancyhdr}
\usepackage[margin=3cm]{geometry}
\usepackage{hyperref}
\hypersetup{
 colorlinks   = true,
 urlcolor     = blue,
 linkcolor    = blue,
 citecolor   = red ,
 bookmarksopen=true
}

\theoremstyle{plain}
\newtheorem{thrm}{Theorem}[section]
\newtheorem{lemma}[thrm]{Lemma}
\newtheorem{prop}[thrm]{Proposition}

\newtheorem{dfn}[thrm]{Definition}

\begin{document}

\newcommand{\SL}{\mathcal L^{1,p}( D)}
\newcommand{\Lp}{L^p( Dega)}
\newcommand{\CO}{C^\infty_0( \Omega)}
\newcommand{\Rn}{\mathbb R^n}
\newcommand{\Rm}{\mathbb R^m}
\newcommand{\R}{\mathbb R}
\newcommand{\Om}{\Omega}
\newcommand{\Hn}{\mathbb H^n}
\newcommand{\aB}{\alpha B}
\newcommand{\eps}{\ve}
\newcommand{\BVX}{BV_X(\Omega)}
\newcommand{\p}{\partial}
\newcommand{\IO}{\int_\Omega}
\newcommand{\bG}{\boldsymbol{G}}
\newcommand{\bg}{\mathfrak g}
\newcommand{\bz}{\mathfrak z}
\newcommand{\bv}{\mathfrak v}
\newcommand{\Bux}{\mbox{Box}}
\newcommand{\e}{\ve}
\newcommand{\X}{\mathcal X}
\newcommand{\Y}{\mathcal Y}
\newcommand{\W}{\mathcal W}
\newcommand{\la}{\lambda}
\newcommand{\vf}{\varphi}
\newcommand{\rhh}{|\nabla_H \rho|}
\newcommand{\Ba}{\mathscr{B}_\gamma}
\newcommand{\Za}{Z_\beta}
\newcommand{\ra}{\rho_\beta}
\newcommand{\na}{\nabla_\beta}
\newcommand{\vt}{\vartheta}

\numberwithin{equation}{section}

\newcommand{\RN} {\mathbb{R}^N}
\newcommand{\Sob}{S^{1,p}(\Omega)}
\newcommand{\Dxk}{\frac{\partial}{\partial x_k}}
\newcommand{\Co}{C^\infty_0(\Omega)}
\newcommand{\Je}{J_\ve}
\newcommand{\beq}{\begin{equation}}
\newcommand{\bea}[1]{\begin{array}{#1} }
\newcommand{\eeq}{ \end{equation}}
\newcommand{\ea}{ \end{array}}
\newcommand{\eh}{\ve h}
\newcommand{\Dxi}{\frac{\partial}{\partial x_{i}}}
\newcommand{\Dyi}{\frac{\partial}{\partial y_{i}}}
\newcommand{\Dt}{\frac{\partial}{\partial t}}
\newcommand{\aBa}{(\alpha+1)B}
\newcommand{\GF}{\psi^{1+\frac{1}{2\alpha}}}
\newcommand{\GS}{\psi^{\frac12}}
\newcommand{\HFF}{\frac{\psi}{\rho}}
\newcommand{\HSS}{\frac{\psi}{\rho}}
\newcommand{\HFS}{\rho\psi^{\frac12-\frac{1}{2\alpha}}}
\newcommand{\HSF}{\frac{\psi^{\frac32+\frac{1}{2\alpha}}}{\rho}}
\newcommand{\AF}{\rho}
\newcommand{\AR}{\rho{\psi}^{\frac{1}{2}+\frac{1}{2\alpha}}}
\newcommand{\PF}{\alpha\frac{\psi}{|x|}}
\newcommand{\PS}{\alpha\frac{\psi}{\rho}}
\newcommand{\ds}{\displaystyle}
\newcommand{\Zt}{{\mathcal Z}^{t}}
\newcommand{\XPSI}{2\alpha\psi \begin{pmatrix} \frac{x}{|x|^2}\\ 0 \end{pmatrix} - 2\alpha\frac{{\psi}^2}{\rho^2}\begin{pmatrix} x \\ (\alpha +1)|x|^{-\alpha}y \end{pmatrix}}
\newcommand{\Z}{ \begin{pmatrix} x \\ (\alpha + 1)|x|^{-\alpha}y \end{pmatrix} }
\newcommand{\ZZ}{ \begin{pmatrix} xx^{t} & (\alpha + 1)|x|^{-\alpha}x y^{t}\\
     (\alpha + 1)|x|^{-\alpha}x^{t} y &   (\alpha + 1)^2  |x|^{-2\alpha}yy^{t}\end{pmatrix}}
\newcommand{\norm}[1]{\lVert#1 \rVert}
\newcommand{\ve}{\varepsilon}

\title[carleman estimates etc.]{  Carleman estimates for Baouendi-Grushin operators with applications to  quantitative uniqueness and strong unique continuation}
\author{Agnid Banerjee}
\address{Tata Institute of Fundamental Research\\
Centre For Applicable Mathematics \\ Bangalore-560065, India}\email[Agnid Banerjee]{agnidban@gmail.com}

\author{Nicola Garofalo}
\address{Dipartimento di Ingegneria Civile, Edile e Ambientale (DICEA) \\ Universit\`a di Padova\\ 35131 Padova, ITALY}
\email[Nicola Garofalo]{rembdrandt54@gmail.com}

\author{Ramesh Manna}
\address{Tata Institute of Fundamental Research\\
Centre For Applicable Mathematics \\ Bangalore-560065, India}
\email[Ramesh Manna]{ramesh@tifrbng.res.in}

\thanks{First author is supported by SERB Matrix grant MTR/2018/000267}
\thanks{ Second author is supported in part by a Progetto SID (Investimento Strategico di Dipartimento) ``Non-local operators in geometry and in free boundary problems, and their connection with the applied sciences", University of Padova, 2017.}
\thanks{Third author is  supported by  SERB National Postdoctoral fellowship, PDF/2017/0027}

\dedicatory{Dedicated to Sergio Vessella, on his $65$-th birthday}

%
%
%

\maketitle

\tableofcontents

\begin{abstract}
In this paper we establish some new $L^{2}-L^{2}$ Carleman estimates for the Baouendi-Grushin operators $\Ba$, in \eqref{pbeta0} below. We apply such estimates to obtain: (i) an extension of the Bourgain-Kenig quantitative unique continuation; (ii) the strong unique continuation property  for some degenerate sublinear equations.

\end{abstract}

\section{Introduction}\label{S:intro}

This paper is devoted to studying some ad hoc $L^{2}-L^{2}$ Carleman estimates with two weights for the Baouendi-Grushin operators 
\begin{equation}\label{pbeta0}
\Ba u= \Delta_z u + |z|^{2\gamma} \Delta_t u,\ \ \ \ \ \ \ \gamma>0.
\end{equation}
We also present applications of such estimates to quantitative uniqueness and strong unique continuation.  
In \eqref{pbeta0} we have indicated with $z\in \R^m$, $t\in \R^k$, the respective variables for the Laplacians, and with $N=m+k$ we will henceforth write $(z,t)\in \RN$. The operator $\Ba$ is degenerate elliptic along the $k$-dimensional subspace $M = \{0\}\times \R^k\subset \R^N$, and it is not translation invariant in $\R^N$. However, it is invariant with respect to translations along $M$. We recall that a more general class of operators modelled on $\Ba$ was first introduced by Baouendi, who studied the Dirichlet problem in weighted Sobolev spaces in \cite{Ba}. Subsequently, Grushin in \cite{Gr1}, \cite{Gr2} analised the hypoellipticity of the operator $\Ba$ when $\gamma \in \mathbb{N}$, and showed that this property is drastically affected  by addition of lower order terms. Pioneering work on a class of subelliptic operators modelled on $\Ba$ was done by Franchi and Lanconelli in \cite{FL0,FL, FL3}. Remarkably, the operator $\Ba$ also played an important role in the recent work \cite{KPS} on the higher regularity of the free boundary in the classical Signorini problem. For the  connection between $\Ba$ and thin obstacle free boundary problems see also \cite{CS} and the more recent work \cite{GRO}. 

When $\gamma = 0$, then we can think of \eqref{pbeta0} as a standard Laplacian. With this observation in mind, the main results in this paper should be viewed as \emph{ad hoc} subelliptic generalisations of a basic Carleman estimate obtained by Escauriaza and Vessella in Theorem 2 of their work \cite{EV}. For the sake of accuracy we mention that, in fact, their result encompasses more general variable coefficient uniformly elliptic equations and even parabolic ones. The elliptic Carleman estimate in \cite{EV} also played a key role in the seminal work \cite{BK} of Bourgain and Kenig on the localization of the Anderson-Bernoulli model (see their Lemma 3.15). These authors obtained for the standard Laplacian a threshold quantitative unique continuation result. For quantitative propagation of smallness we refer the reader to the work \cite{MV}. 

In this paper we generalise the quantitative estimate of Bourgain and Kenig to equations of the form
 \begin{equation}\label{e0}
- \Ba u =V(z,t) u,
\end{equation} 
where $\Ba$ is as in \eqref{pbeta0}. In a third direction, we study sublinear equations of the type
\begin{equation}\label{e1}
-\Ba u =V(z,t) u + f((z,t),u),
\end{equation}
with appropriate assumptions on $f$, see the discussion below.

To provide the reader with some perspective we mention that when $\gamma = 1$ the operator $\Ba$ is intimately connected to the sub-Laplacians in groups of Heisenberg type. In such Lie groups, in fact, in the exponential coordinates with respect to a fixed orthonormal basis of the Lie algebra, the sub-Laplacian  is given by
\begin{equation}\label{slH}
\Delta_H = \Delta_z + \frac{|z|^2}{4} \Delta_t  + \sum_{\ell = 1}^k \p_{t_\ell} \sum_{i<j} b^\ell_{ij} (z_i \p_{z_j} -  z_j \p_{z_i}),
\end{equation}
where $b^\ell_{ij}$ indicate the group constants, see e.g. \cite{Gmem}. If $u$ is a solution of $\Delta_H$ that further annihilates the symplectic vector field $\sum_{\ell = 1}^k \p_{t_\ell} \sum_{i<j} b^\ell_{ij} (z_i \p_{z_j} -  z_j \p_{z_i})$, then, up to a normalisation factor of $4$, $u$ solves  the operator $\Ba$ obtained by letting $\gamma = 1$ in \eqref{pbeta0} above. We mention, in this connection, the remarkable fact  that even the weak unique continuation property fails for $-\Delta_H + V$, see \cite{Bao}. For some positive results in the Heisenberg group, and in general Carnot groups, see however \cite{GLan} and \cite{GR}.

Concerning the question of interest for this paper, the unique continuation property, we mention that for general uniformly elliptic equations there are essentially two known methods  for proving it. The former  is  based on Carleman inequalities, which are appropriate weighted versions of  Sobolev-Poincar\'e inequalities. This method  was first introduced by T. Carleman in his fundamental work \cite{C} in which  he  showed that  strong unique continuation holds  for equations of the type  $-\Delta u +V u = 0$, with $V \in L^{\infty}_{loc}(\R^2)$.
Subsequently, his estimates were generalised (and extended) to higher dimensions in \cite{Mu, He, HW, Co}, and to uniformly elliptic operators with $C^{2,\alpha}_{loc}$ principal part in the pioneering work \cite{A}. The results in \cite{A} were further extended to $C^{0,1}_{loc}$ coefficients in \cite{AKS}. 
We recall that unique continuation fails in general when the coefficients of the principal part are only H\"older continuous, see Plis' famous counterexample in \cite{Pl} for non-divergence equations, and also \cite{Mi} for equations in divergence form. 
The second approach came up in the works of Lin and the second named author, see \cite{GL1}, \cite{GL2}. Their method is based on the almost monotonicity of a generalisation of the frequency function, first introduced by Almgren in \cite{Al} for harmonic functions. Using this approach, they were able to obtain new quantitative  information  for the solutions to  divergence form elliptic equations with Lipschitz coefficients which in particular encompass and improve on those in \cite{AKS}.

The unique continuation property for the degenerate operators $\Ba$ is much subtler than the one for the Laplacian. The reason for this is that the operator $\Ba$ does not preserve ``radial" functions, see \eqref{ii} below. As a consequence, when inverting the operator $-\Ba + V$ one is confronted with Carleman estimates in which terms with both positive and negative powers of the singular ``angle function" $\psi$ in \eqref{psi} appear, and this complicates matters considerably. As the reader will see, the proofs of  our main results Theorems \ref{thm2}, \ref{main} and \ref{main4} exploit several non-trivial geometric facts that beautifully combine.

We mention that when $V\equiv 0$ in \eqref{e0} the strong unique continuation was first established by the second named author in  \cite{G}. In that work he introduced a Almgren type  frequency function associated with $\Ba$, and proved that such function is monotone non-decreasing on solutions of $\Ba= 0$. More in general, in \cite[Theorem 4.2]{G} an almost monotonicity result was proved for equations such as $\Ba u = <b(z,t),\nabla u> + V(z,t) u$, with possibly singular lower-order terms satisfying the following threshold conditions in a neighbourhood of the origin 
\begin{equation}\label{G}
|V(z,t)| \le C \frac{f(\rho)}{\rho^2} \psi, \ \ \ \ \ \ \ \ |<b(z,t),\nabla u>|\le C |Xu|\frac{f(\rho)}{\rho} \psi^{1/2},
\end{equation}
where for the notation $Xu, \rho$ and $\psi$ we respectively refer the reader to \eqref{Xgrad}, \eqref{rho} and \eqref{psi} below. In \eqref{G}, $f:(0,R_0)\to (0,\infty)$ is a non-decreasing function satisfying the Dini integrability assumption $\int_0^{R_0} \frac{f(t)}{t} dt <\infty$. This includes as a special case the example $f(t) = t^\delta$, $\delta>0$.
These results were extended to more general variable coefficient equations by Vassilev and the second named author in \cite{GV}. One should also see the related works \cite{GLan} and \cite{GR} on the Heisenberg and more general Carnot groups. We also note that a version of the monotonicity formula for $\Ba$ played an extensive role in the recent work \cite{CSS} on the obstacle problem for the fractional Laplacian. 

Using some ad hoc Carleman estimates with two weights, in \cite{GarShen} the authors were able to establish for the first time some strong unique continuation results for \eqref{e0} in the difficult situation when $V$ satisfies appropriate $L^p$ integrability hypothesis. Their analysis, which is closer in spirit to the works \cite{JK}, \cite{J}, \cite{ChS}, \cite{KT} to name a few,  only covers the special case when $\gamma=k=1$ in \eqref{pbeta0}, and ultimately rests on delicate boundedness properties of certain projector operators generalising some of the results in \cite{So}.  We also refer to the recent work \cite{BM} where, using such projector operator estimates, a new $L^{2}-L^{2}$ Carleman estimate is derived. Using the latter, the authors deduce strong unique continuation when the potential $V$  satisfies Hardy type growth assumptions. It is worth mentioning at this point that the general situation of the results in \cite{GarShen} presently remains a challenging open question, even in the simplest case when $V\in L^\infty_{loc}(\R^N)$!

After this preliminary discussion we turn to the results in the present paper. Our first contribution is the following new $L^{2}-L^{2}$ type Carleman estimate for $\Ba$ with singular weights. Such result, Theorem \ref{thm2}, should be viewed as one of the key steps in the proofs of the two subsequent main Carleman estimates in this paper, Theorems \ref{thm3} and \ref{sub1}. We have decided to present it separately in order to provide the reader with some of the main ideas in a simplified, yet significant, situation. 
The expert reader will recognise that the estimate \eqref{est1} below generalises (and in fact encompasses) the following classical one  for the standard Laplacian, see \cite{C, Mu, He, HW, Co, A, AKS, Ho}:
there exists $C=C(n)>0$ such that for every $\alpha >0$ sufficiently large, and $u \in C_0^{\infty}(\Rn \setminus \{0\})$, with $\operatorname{supp} u\subset \{|x|<R\}$, one has
\begin{eqnarray}\label{ar}
\int |x|^{-2\alpha}(u^2+|\nabla u|^2)dx \leq CR^2\int |x|^{-2\alpha}(\Delta u)^2 dx.
\end{eqnarray}
It is well known that \eqref{ar} implies the strong unique continuation property for perturbations of the Laplacian, with lower order terms satisfying  appropriate bounds. In our setting, adapting the arguments for the case of elliptic operators in \cite{AKS}, one can deduce starting from Theorem \ref{thm2} the strong unique continuation for solutions of the differential inequality
\begin{equation}\label{aks}
|\Ba u| \leq   C_1 \frac{\psi}{\rho^{2-\delta}} |u| + C_2 \frac{\psi^{1/2}}{\rho^{1-\delta}} |Xu|,
\end{equation}
for some $\delta >0$, and $C_1, C_2 \ge 0$, although, as we have mentioned in the discussion of the above hypothesis \eqref{G}, this result is already contained  in \cite[Theorems 4.2, 4.3 \& Corollary 4.3]{G}. We also recall that for the standard Laplacian it is known from the works \cite{AB, Wo, Reg, Gra, PW} that it is possible to allow $\delta = 0$ in \eqref{aks}, provided that $C_2$ is sufficiently small. We presently do not know whether this extends to the setting of the present work.

While for the relevant notation we refer the reader to Section \ref{pre}, one should note the appearance of the different weights $\psi$ and $\psi^{-1}$ in the two sides of \eqref{est1} below. In this inequality, and henceforth, we omit indicating the Lebesgue measure $dz dt$ in the relevant integrals. Also, the number $Q$ will always indicate the homogeneous dimension \eqref{Q} associated with the non-isotropic dilations \eqref{dil}.

\begin{thrm} \label{thm2}
For every $\ve>0$, there exists $C=C(m, k, \ve)>0$ such that for every $\alpha > \max\{0,(Q-4)/2\}$, $R>0$, and $u \in S^{2,2}_{0}(B_R \setminus \{0\})$ with  $\operatorname{supp} u \subset B_R \setminus \{0\}$, one has
\begin{align}\label{est1}
&\alpha^2 \int_{B_R} \rho^{-2\alpha-4+\ve}u^2 \psi + \int_{B_R}  \rho^{-2\alpha-2+\ve} |Xu|^2  \leq C R^{\ve}  \int_{B_R} \rho^{-2\alpha}(\Ba u)^2 \psi^{-1}.
\end{align}
 \end{thrm}
 
The Carleman  estimate \eqref{est1} is new for the general Baouendi-Grushin operators \eqref{pbeta0}. Only when $\gamma=1$ and $k=1$ the estimate follows from that in Theorem 3.2 in \cite{BM} which also  holds for $\ve=0$. However, as we have mentioned above, the proof in \cite{BM} relies on delicate $L^{2}-L^{2}$ projection estimates previously established  in \cite{GarShen}, whereas \eqref{est1} will be derived from fairly elementary considerations exploiting some remarkable geometric properties of the operator first noted in \cite{G}, see also \cite{GV} and \cite{GR}. We mention that the reason for which we cannot take $\ve=0$ in \eqref{est1} is the failure for $\ell= Q$ of the Hardy inequality \eqref{h1} in Lemma \ref{hardy}.

A second question we study is that of quantitative uniqueness for stationary Schr\"odinger equations  of the type \eqref{e0}, where the potential $V$ satisfies a boundedness assumption as in \eqref{vasump} below. In order to provide some context to this aspect, we recall that in their seminal work \cite{BK} Bourgain and Kenig showed that non-trivial solutions to $\Delta u = Vu$
have vanishing order proportional to $||V||_{L^{\infty}}^{2/3} +1$. This is sharp in view of Meshov's counterexample in \cite{Me}.  In Theorem \ref{main} below we extend this result to the operator \eqref{e0}. The key ingredient is the following refined Carleman estimate, which generalises that in \cite{EV} and \cite{BK}. 

\begin{thrm} \label{thm3}
Let $0< \ve < 1$ be fixed. There exists a universal $R_0>0$, depending on $\ve$, such that  for $R \leq R_0$, $u \in S^{2,2}_{0}(B_R \setminus B_{aR})$,  and $V$ satisfying \eqref{vasump}, one has \begin{align}\label{est}
&\alpha^3\int \rho^{-2\alpha- 4+\ve} u^2 e^{2\alpha \rho^{\ve}} \psi + \alpha^2 \int \rho^{-2\alpha -4} u^2 e^{2\alpha \rho^{\ve}} \psi \leq  C \int \rho^{-2\alpha} e^{2\alpha \rho^{\ve}} (\Ba u +V u)^2 \psi^{-1},
\end{align}
for constants $C, C_1>0$ depending on $m, k, a, \ve$, and for all $\alpha > \frac{Q-4}2$, such that also $\alpha \geq  C_1 (K^{2/3} +1)$.
 \end{thrm}

Using the estimate in Theorem \ref{thm3} above,  we obtain   the following  subelliptic quantitative uniqueness result.

\begin{thrm}\label{main}
Let $u \in S^{2,2}(B_R)$ with $|u| \leq K_0$ for some $K_0>0$ be a solution to \eqref{e0} in $B_{R}$, with $V$ satisfying the bound
\begin{equation}\label{vasump}
|V(z,t)| \leq K \psi.
\end{equation}
 Then, there exists  $R_0 \in (0, R/2]$  depending on $m, k, \gamma$, and  constants $C_1,C_2>0$ depending on $m, k, \gamma,   K_0$ and $\int_{B_{\frac{R_0}{4}}}  u^2 \psi$, such that for all $0<r<  \frac{R_0}{9}$ one has
\begin{equation}\label{main1}
||u||_{L^{\infty}(B_r)} \geq C_1 \left(\frac{r}{R_0}\right)^{C_2 (K^{2/3}+1)}.   
\end{equation}
\end{thrm}

It is worth emphasizing that, when $k=0$, we have $N = m$ and then from \eqref{psi} we have $\psi \equiv 1$. In this case the constant $K$ in \eqref{vasump} can be taken to be  $||V||_{L^{\infty}}$, and therefore Theorem \ref{main} reduces to the cited Euclidean result in \cite{BK}. We also note that when $V$ satisfies the additional hypothesis
\[
|ZV| \leq K\psi,
\]
then, using a variant of the frequency function approach, the following sharper estimate was established in \cite{BG1} for solutions to \eqref{e0},
\begin{equation}\label{m}
||u||_{L^{\infty}(B_r)} \geq C_1 \left(\frac{r}{R_0}\right)^{C_2 (\sqrt{K}+1)}.   
\end{equation}

The reader should note that  for Laplacian on a compact manifold the counterpart of \eqref{m} was first obtained using Carleman estimates by Bakri in \cite{Bk}. This generalised the sharp vanishing order estimate of Donnelly and Fefferman in \cite{DF1} for eigenfunctions of the Laplacian. We also mention that, for the standard Laplacian, the result of Bakri was subsequently obtained by Zhu \cite{Zhu1}, using a variant of the frequency function approach in \cite{GL1, GL2}. This was extended in \cite{BG} to more general elliptic equations with Lipschitz principal part. The authors also established a certain boundary version of the vanishing order estimate. Finally, we refer to the  paper \cite{Ru1} for an interesting generalisation to nonlocal equations of the quantitative uniqueness result in \cite{Bk}, and also to \cite{B} for a generalisation to Carnot groups of arbitrary step.

As a final application of the Carleman estimates, in this paper we present a generalisation to the subelliptic operator $\Ba$ of the strong unique continuation results in \cite{SW} and \cite{Ru} for sublinear equations. Precisely, we consider equations of the type
 \begin{equation}\label{sub}
 -\Ba u=  f((z,t), u) \psi + Vu, 
 \end{equation}
 where $V$ satisfies  the growth  condition \eqref{vasump}, and the nonlinearity $f$ and its primitive, $F((z,t),  s) = \int_{0}^{s} f((z,t) s) ds$, satisfy the following structural assumptions:   
 \begin{equation}\label{a1}
 \begin{cases}
 f((z,t), 0) =0,
 \\
  0 < s f((z,t), s) \leq q F((z,t), s), \ \text{for some $q \in (1, 2)$ and  $s \in (-1, 1) \setminus \{0\}$},
  \\
  |\nabla_{(z,t)} f| \leq K |f|,\ |\nabla_{(z,t)} F| \leq K F,
  \\
  f((z,t), s) \leq \kappa s^{p-1}\ \text{for some $p \in (1,2)$}.
\end{cases}
\end{equation}
 We note that the first  and the last condition  in \eqref{a1} imply that  for constants $c_0, c_1$, we have 
\begin{equation}\label{a0}
 c_1 s^{p} \geq  F(\cdot,s) \geq c_0 s^{q},\ \text{for}\ s \in (-1,1).
  \end{equation}
  A  prototypical $f$ satisfying \eqref{a1} is $f((z,t), u) = \sum_{i=1}^\ell c_i(z,t) |u|^{q_i-2} u$, where for some constants $k_0, k_1$ and $K$, one has for each $i = 1,...,\ell$: $q_i \in (1,2)$, $0<k_0<c_i< k_1$, and $|\nabla c_i| < K$.  In this case, we can take  $q= \max \{q_i\}$ and $p= \min \{q_i \}$. With $f$ satisfying the hypothesis \eqref{a1}, the unique continuation property for uniformly elliptic sublinear equations of the type 
\[
-\operatorname{div}(A(x) \nabla u) = f(x, u)
\]
was recently studied in the interesting papers \cite{SW} and \cite{Ru}. More precisely, the weak unique continuation for such  sublinear equations 
was first obtained in \cite{SW} using the frequency function approach. Subsequently, R\"uland  in \cite{Ru}  established the strong unique continuation property for such equations using Carleman estimates. We note that the sign assumption on $f, F$ in \eqref{a1} is not restrictive because the strong unique continuation property fails when $f= -|u|^{q-2}u$ and $A=\mathbb{I}$ (for a counterexample in the one-dimensional situation, see e.g. \cite{SW}).
The next theorem extends R\"uland's result to degenerate elliptic equations of the type \eqref{sub}. We also note that, when specialised to the case of the standard Laplacian, our proof slightly   simplifies that in \cite{Ru}.

\begin{thrm}\label{main4}
Assume that $V$ be as in \eqref{vasump}, and that $f$ satisfy \eqref{a1}. Let $u \in S^{2,2}(B_1)$ be a solution to \eqref{sub} in $B_1$ such that $||u||_{L^{\infty}(B_1)} \leq 1$. If $u$ vanishes to infinite order at $(0,0)$ in the sense of Definition \ref{v0}, then $u \equiv 0$. 
\end{thrm}

In closing, we briefly describe the organisation of the paper. In Section 2, we introduce  the relevant notions and gather some results from \cite{G} which will be needed in the rest of the paper. In Section 3 we prove Theorem \ref{thm2}. Section \ref{S:qu} is devoted to the proof of Theorems \ref{thm3} and \ref{main}. The final Section \ref{S:sub} contains the proof of Theorem \ref{main4}. 


\section{Notations and preliminary results}\label{pre}

Henceforth in this paper we follow the notations adopted in \cite{G} and \cite{GV}, with one notable proviso: the parameter $\gamma>0$ in \eqref{df}, etc. in this paper plays the role of $\alpha >0$ in \cite{G} and \cite{GV}. The reason for this is that we have reserved the greek letter $\alpha$ for the powers of the singular  weights in our Carleman estimates.  Throughout the paper, whenever convenient, we will use the summation convention over repeated indices. Let $N = m + k$, and denote an arbitrary point in $\RN$ as  $(z,t) \in \R^m \times \R^k$. If we consider the vector fields $X_1,...,X_N$ in $\R^N$ defined by 
\begin{align}\label{df}
& X_i= \partial_{z_i},\ \ \  i=1, ...m,\ \ \ \  \ \ \ \ X_{m+j}= |z|^{\gamma} \partial_{t_j},\ \ \   j=1, ...k,
\end{align}
then it is immediate to recognise that
\begin{equation}\label{pbeta}
\Ba u= \sum_{i=1}^N X_i^2 u.
\end{equation}
Given a function $u$, we respectively denote by
\begin{equation}\label{Xgrad}
Xu= (X_1 u,...,X_N u),\ \ \ \ \ \ \ \ \ |X u|^2= <X u,X u> = \sum_{i=1}^N (X_i u)^2,
\end{equation}
the intrinsic (degenerate) gradient of a function $u$, and the square of its length. 
We  now define the relevant function space for our work. The reader should bear in mind that, henceforth, we routinely omit indicating in all integrals the Lebesgue measure $dz dt$ in $\R^N$. 

\begin{dfn}
Given an open set $\Om\subset \R^N$,  we denote by  $S^{2,2}(\Om)$ the completion of $C^{\infty}(\overline{\Om})$ with respect to the norm
\[
||f||_{S^{2,2}(\Om)}= \int_{\Om} f^2 + |Xf|^2 + \sum_{i,j=1}^N |X_i X_j f|^2.  
\]
We instead indicate with $S^{2,2}_0(\Om)$ the completion of $C^{\infty}_{0}(\Om)$ with respect to the same norm.
\end{dfn}
We note that the vector fields $X_i$ are homogeneous of degree one with respect to the following family of anisotropic dilations 
\begin{equation}\label{dil}
\delta_\la(z,t)=(\la z,\la^{\gamma+1} t),\ \ \ \ \ \ \ \ \la>0.
\end{equation}
Consequently, $\Ba$ is homogeneous of degree two with respect to \eqref{dil}. 
The infinitesimal generator of the family of dilations \eqref{dil}  is given by the vector field
\begin{equation}\label{Z}
Z= \sum_{i=1}^m z_i \partial_{z_i} + (\gamma+1)\sum_{j=1}^k t_j \partial_{t_j}.
\end{equation} 
A  function $v$ is $\delta_{\la}$-homogeneous of degree $\kappa$ if and only if $Zv=\kappa v$.  We note that Lebesgue measure scales with respect to the anisotropic dilations \eqref{dil} according to the formula $d(\delta_\la(z,t)) = \la^Q dz dt$, where
\begin{equation}\label{Q}
Q= m + (\gamma+1) k.
\end{equation}
Consequently, the number $Q$ plays the role of a dimension in the analysis of the operator
$\Ba$ (since $m, k\ge 1$ and $\gamma>0$, we note that $Q>2$).
For instance, one has the remarkable fact, discovered in \cite{G}, that the fundamental solution $\Gamma$ of $\Ba$  with pole at the origin is given by the formula
\[
\Gamma(z,t) = \frac{C}{\rho(z,t)^{Q-2}},\ \ \ \ \ \ \ \ \ (z,t)\not= (0,0),
\]
where $C>0$ is suitably chosen, and $\rho$ is the pseudo-gauge 
\begin{equation}\label{rho}
\rho(z,t)=(|z|^{2(\gamma+1)} + (\gamma+1)^2 |t|^2)^{\frac{1}{2(\gamma+1)}}.
\end{equation}
We respectively denote by 
\[
B_r = \{(z,t)\in \R^N\mid \rho(z,t) < r\},\ \ \ \ \ \ \ \ S_r = \{(z,t)\in \R^N\mid \rho(z,t) = r\}, 
\]
the gauge pseudo-ball and sphere centred at $0$ with radius $r$. Since $\rho$ in \eqref{rho} is homogeneous of degree one, we have
\begin{equation}\label{hg}
Z\rho=\rho.
\end{equation}
 We also need the angle function $\psi$ introduced in \cite{G}
\begin{equation}\label{psi}
\psi = |X\rho|^2= \frac{|z|^{2\gamma}}{\rho^{2\gamma}}.
\end{equation}
The function $\psi$ vanishes on the characteristic manifold $M=\Rn \times \{0\}$, and clearly satisfies $0\leq \psi \leq 1$. Since $\psi$ is homogeneous of degree zero with respect to \eqref{dil}, one has
\begin{equation}\label{Zpsi}
 Z\psi = 0.
 \end{equation}
 If $f\in C^2(\R)$ and $v\in C^2(\R^N)$, then we have the important identities (see \cite{G}):
 \begin{equation}\label{ii}
 \Ba f(\rho) = \psi \left(f''(\rho) + \frac{Q-1}{\rho} f'(\rho)\right),
 \end{equation}
 and 
 \begin{equation}\label{h10}
<Xv,X\rho> = \sum_{i=1}^N X_i v X_i \rho = \frac{Zv}{\rho} \psi.
\end{equation}
Henceforth, for any two vector fields $U$ and $W$, $[U,W] = UW - WU$ denotes their commutator. In the next proposition we collect some  important identities from \cite{G}. 

\begin{prop}\label{Est}
The following identities hold.
\begin{itemize}
\item[(i)] $ \operatorname{div} Z=Q$;
\item[(ii)] $[X_i,Z]u = X_i u \ \ \ \ i=1,...,N$;
\item[(iii)] $\operatorname{div} ( \rho^{-\ell} Z)= (Q-\ell) \rho^{-\ell}$;
\item[(iv)] $\operatorname{div} ( \rho^{-Q} \log \rho\ Z)=  \rho^{-Q}$;
\item[(iv)] $\operatorname{div} X_i=0$.
\end{itemize}
\end{prop}

We now state the relevant  Hardy type inequality which is crucially needed in the  proof of  our Carleman estimates. It follows from the proof of Corollary 4.1 in \cite{G}. 
\begin{lemma}\label{hardy}
Let $u \in C_0^{\infty}(\R^N\setminus\{0\}).$ For every $\ell \in \R,$ with $\ell \ne Q,$ one has 
\begin{eqnarray}\label{h1}
\int \frac{u^2}{\rho^{\ell}} \psi \leq\left(\frac{4}{\ell-Q}\right)^2 \int \frac{(Zu)^2}{\rho^{\ell}} \psi.
\end{eqnarray}
If instead $\ell = Q$ we obtain
\begin{equation}\label{h2}
\int \frac{u^2}{\rho^{Q}} \psi \leq 4 \int \frac{(\log \rho)^2}{\rho^{Q}}(Zu)^2 \psi.
\end{equation}
 \end{lemma}

\begin{proof}
Using (iii) in Proposition \ref{Est} and the divergence theorem, we obtain  
\begin{align*}
& \int u^2 \rho^{-\ell} \psi = \frac{1}{Q-\ell} \int \operatorname{div}(\rho^{-\ell} Z) \psi u^2  
\\
& = \frac{1}{Q-\ell} \int \operatorname{div}(\psi u^2  \rho^{-\ell} Z)  - \frac{1}{Q-\ell} \int  \rho^{-\ell} Z(\psi u^2) 
\\
& = - \frac{2}{Q-\ell} \int u Zu \psi \rho^{-\ell}  -  \frac{1}{Q-\ell} \int u^2 Z \psi \rho^{-\ell} 
\\
& \le \frac{2}{Q-\ell} \left(\int u^2  \rho^{-\ell} \psi\right)^{1/2}  \left(\int (Zu)^2 \psi \rho^{-\ell}\right)^{1/2},
\end{align*}
where we have used \eqref{Zpsi}. The numerical inequality $2ab \le \ve a^2 + \ve^{-1} b^2$, with the choice $\ve = (Q-\ell)/2$, easily implies the desired conclusion \eqref{h1} when $\ell\not= Q$. If $\ell = Q$, we use instead (iv) in Proposition \ref{Est} and we argue similarly to the previous case in order to obtain \eqref{h2}.

\end{proof}

We note that the inequality \eqref{h1} in Lemma \ref{hardy} fails for $\ell=Q$ with a constant in the right-hand side which is independent of the function $u$. To see this, let $\zeta\in C^\infty_0(0,\infty)$, $0\le \zeta\le 1$, $\zeta \equiv 0$ for $0\le r\le 1/2$, $\zeta \equiv 1$ for $r\ge 1$, and consider the functions $u_\ve(z,t) = \zeta(\rho(z,t)/\ve)$, $\ve\in (0,1)$. If we insert such functions in the inequality in Lemma \ref{hardy}, observing that $Z u_\ve = \ve^{-1} Z\rho \zeta'(\rho/\ve) = \ve^{-1} \rho \zeta'(\rho/\ve)$, we see that the integral in the right-hand side is bounded above by
\[
\frac{C}{\ve^2} \int_{\frac{\ve}2 \le \rho\le \ve} \frac{\psi}{\rho^{Q-2}} \le C^\star,
\]
where $C, C^\star>0$ are absolute constants independent of $\ve$. On the other hand, the integral in the left-hand side is bounded below by
\[
\int_{\ve \le \rho\le 1} \frac{\psi}{\rho^{Q}}  = \ve^Q \int_{\frac{1}\ve \ge \rho\le 1} \frac{\psi}{\rho^{Q}}  = \ve^Q \int_1^{\frac 1{\ve}} r^{-Q} \int_{\p B_r} \frac{\psi}{|\nabla \rho|} dH_{N-1} dr,
\]
where in the last equality we have used Federer's coarea formula. Now, a scaling argument, and the fact that $\psi$ is $\delta_\la$-homogeneous of degree zero, show that 
\[
\int_{\p B_r} \frac{\psi}{|\nabla \rho|} dH_{N-1} = \sigma r^{Q-1},
\]
where $\sigma>0$ is an absolute constant. We thus conclude that as $\ve\to 0^+$, 
\[
\int_{\ve \le \rho\le 1} \frac{\psi}{\rho^{Q}}  = \sigma \log(1/\ve) \to + \infty,
\]
This shows that \eqref{h1} in Lemma \ref{hardy} cannot possibly hold when $\ell = Q$.
 However, the following weaker inequality is true. We will need it in the proof of our refined estimate in Theorem \ref{thm3}. 

\begin{lemma}\label{hardy2}
Let $u \in C_0^{\infty}(B_R \setminus B_{aR})$ for some $0<a<1$.  Then, the following inequality holds
\begin{eqnarray*}
\int \frac{u^2}{\rho^{Q}} \psi \leq \frac{4}{a} \int \frac{(Zu)^2}{\rho^{Q}} \psi.
\end{eqnarray*}
\end{lemma}

\begin{proof}
Using \eqref{h1} with $\ell=Q+1$, and the fact that $u$ is supported in $B_R \setminus B_{aR}$,  we have 
\begin{align*}
& \int \frac{u^2}{\rho^{Q}} \psi  \leq  R \int \frac{u^2}{\rho^{Q+1}} \psi
\leq 4 R \int \frac{(Zu)^2}{\rho^{Q+1}} \psi 
 \leq \frac{4R}{aR} \int \frac{(Zu)^2}{\rho^{Q}} \psi = \frac{4}{a} \int \frac{(Zu)^2}{\rho^{Q}} \psi, 
\end{align*}
where in the last inequality we have used the simple observation that $\frac{1}{\rho} \leq \frac{1}{aR}$.

\end{proof}

We close this section by introducing the relevant notion of vanishing to infinite order.

\begin{dfn}\label{v0}
  We  say  that $u$ vanishes to infinite order at the origin if for every $\ell>0$ one has as $r \rightarrow 0$,
 \begin{align}\label{vanLp}
 \int_{B_r} |u|^2 \psi  = O(r^\ell).
 \end{align}
 \end{dfn}


 \section{Proof of Theorem \ref{thm2}}

In this section we prove Theorem \ref{thm2}. Our plan is to first obtain a suitable bound from below for the right-hand side of \eqref{est1}, see \eqref{rt4}, and then establish a bound from above of the left-hand side in terms of the same quantity. First, by a limiting argument, it suffices to establish the estimate \eqref{est1} when $u$ is smooth and $\operatorname{supp} u \subset B_R\setminus \{0\}$. We define $v=\rho^{-\beta}u$, so that  $u=\rho^{\beta}v,$ with $\beta$  is to be determined later (depending on $\alpha$ and $Q$). Then, we have  
$$\Ba u=v\Ba (\rho^{\beta})+2 <X \rho^{\beta},X v>  + \Ba v \rho^{\beta}.$$
Now, we calculate the first two terms in the right-hand side of the above equation. By \eqref{ii} one has
\begin{eqnarray*}
\Ba(\rho^{\beta})= \left(\beta(\beta-1)\rho^{\beta-2}+(Q-1) \beta \rho^{\beta-2}\right) \psi
= \beta(\beta+Q-2)\rho^{\beta-2} \psi.
\end{eqnarray*}
Using instead \eqref{h10} we obtain
\[
2\sum_{i=1}^N X_i \rho^{\beta} X_i v =2\beta \rho^{\beta-2} \psi Zv.
\]
Therefore,
\begin{eqnarray}\label{b1}
\Ba u&= v\psi\left(\beta(\beta+Q-2)\rho^{\beta-2}\right)+2\beta \rho^{\beta-2} \psi Zv+ \Ba v \rho^{\beta}.
\end{eqnarray}
Using  $(a+b)^2\geq a^2+2ab$, with $a= 2\beta \rho^{\beta-2} \psi Zv$, and with  $b$ being the  rest of the terms in \eqref{b1},  we  obtain 
\begin{eqnarray*}
(\Ba u)^2\geq 4\beta^2 \rho^{2\beta-4} \psi^2 (Zv)^2
+ 4 \beta \rho^{\beta-2} \psi Zv\left(\psi(\beta(\beta+Q-2))\rho^{\beta-2}v+\Ba v \rho^{\beta}\right).
\end{eqnarray*}
We now integrate the latter inequality with respect to the measure $\rho^{-2\alpha} \psi^{-1} dz dt$ on $\R^N$, keeping in mind that, henceforth, we routinely omit indicating in all integrals the domain of integration $B_R$ and the Lebesgue measure $dz dt$. We thus have
\begin{align}\label{rt1}
&\int \rho^{-2\alpha}(\Ba u)^2  \psi^{-1} \geq \int 4\beta^2 \rho^{2\beta-2\alpha-4} (Zv)^2 \psi 
\\
& + \int 4 \beta^2 (\beta+Q-2) \rho^{2\beta-2\alpha-4} v Zv\psi +\int 4\beta  \rho^{2\beta-2\alpha-2} Zv \Ba v  
\notag\\ 
&= 4\beta^2 \int   \rho^{2\beta-2\alpha-4} (Zv)^2  \psi  
+ 2 \beta^2 (\beta+Q-2) \int  \rho^{2\beta-2\alpha-4} Z(v^2) \psi
\notag\\
& + 4\beta \int   \rho^{2\beta-2\alpha-2} Zv  \Ba v.
\notag
\end{align}
We now choose $\beta = \frac{2\alpha+4 - Q}2$, which gives $2\beta - 2 \alpha - 4 = - Q$. We are going to be interested exclusively in values of $\beta >0$, which amounts to taking $\alpha > \frac{Q-4}2$. With such choice we obtain
\[
\int \rho^{2\beta-2\alpha-4} Z(v^2) \psi = \int \rho^{-Q} Z(v^2) \psi.
\]
On the other hand, \eqref{Zpsi} and (iii) in Proposition \ref{Est}, give
\[
\operatorname{div}(\rho^{-Q} \psi v^2 Z) = \rho^{-Q} Z(v^2) \psi.
\]
We thus find
\begin{equation}\label{zero}
\int \rho^{2\beta-2\alpha-4} Z(v^2) \psi = 0.
\end{equation}
Next, we note that with our choice of $\beta$ the last term in the right-hand side of \eqref{rt1} becomes
\begin{equation}\label{rt2}
4\beta  \int  \rho^{2\beta-2\alpha-2} Zv  \Ba v  = 4\beta  \int \rho^{-Q+2} Zv \Ba v.
\end{equation}
In order to estimate the integral in the right-hand side we  use the following Rellich type identity in Lemma 2.11 in \cite{GV}:
\begin{align*}
& \int_{\partial B_R} |Xv|^2<\mathscr Z,\nu>  = 2 \int_{\partial B_R} X_i v <X_i,\nu> \mathscr Z v
\\
& -  2\int_{B_R} (\operatorname{div} X_i) X_iv \mathscr Z v  - 2 \int_{B_R}  X_iv [X_i,\mathscr Z]v
\notag\\
& + \int_{B_R} \operatorname{div} \mathscr Z |Xv|^2     - 2 \int_{B_R} \mathscr Z v  \Ba v,
\notag
\end{align*}
where $\mathscr Z$ is any smooth vector field. Applying this identity with the choice $\mathscr Z =\rho^{-Q+2}Z$, noting that since $v$ is compactly supported in $B_R \setminus \{0\}$ the boundary terms do not appear, that from (iv) in Proposition \ref{Est} we have $\operatorname{div} X_i=0$, and that from (iii) we have $\operatorname{div} (\rho^{-Q+2}Z) = 2 \rho^{-Q+2}$, we conclude 
\begin{align}\label{ok11}
& 4 \beta \int \rho^{-Q+2} Zv \Ba v =2 \beta  \int \operatorname{div} (\rho^{-Q+2}Z) |Xv|^2  -4 \beta \int X_i v [X_i, \rho^{-Q+2} Z] v
\\
& = 4\beta \int \rho^{-Q+2} |Xv|^2  -4 \beta \int X_i v [X_i, \rho^{-Q+2} Z] v.
\notag
\end{align}
Next, by (ii) in Proposition \ref{Est} we have 
\[
[X_i,\rho^{-Q+2}Z]v = \rho^{-Q+2}[X_i, Z]v+X_i(\rho^{-Q+2})Zv=\rho^{-Q+2}X_iv+(2-Q) \rho^{-Q+1}X_i(\rho)Zv.
\] 
Combining this observation with \eqref{h10}, we find
\begin{align}\label{ok1}
& -4 \beta \int X_i v [X_i, \rho^{-Q+2} Z] v= - 4 \beta \int \rho^{-Q+2} |Xv|^2
\\
& + 4 \beta (Q-2) \int \rho^{-Q} (Zv)^2 \psi.
\notag
\end{align}
Substituting this conclusion in \eqref{ok11}, we have
\begin{equation} \label{rt3}
4 \beta \int \rho^{-Q+2} Zv \Ba v  = 4 \beta (Q-2) \int \rho^{-Q} (Zv)^2 \psi \geq 0,
\end{equation}
where the last inequality follows from the fact that $Q > 2$ and $\beta>0$.
Combining \eqref{rt1}, \eqref{zero} and \eqref{rt3}, we finally obtain that 
\begin{equation}\label{rt4}
\int_{B_R} \rho^{-2\alpha}  (\Ba u)^2 \psi^{-1} \geq 4 \beta^2 \int_{B_R} \rho^{-Q}  (Zv)^2 \psi. 
\end{equation}

On the other hand, recalling our choice $2\alpha - 2 \beta + 4 = Q$,  and applying the Hardy inequality \eqref{h1} in Theorem \ref{hardy} with $\ell = Q-\ve$, we see that the first term in the left-hand side of \eqref{est1}   can be controlled from above in the following way, 
\begin{align*}
\alpha^2 \int \rho^{-2\alpha-4+\ve} u^2 \psi & = \alpha^2 \int \rho^{-2\alpha-4+\ve} \rho^{2 \beta}  v^2 \psi = \alpha^2 \int \rho^{-Q+\ve}  \psi v^2
\\
& \leq \frac{16 \alpha^2}{\ve^2}\int \rho^{-Q+\ve} (Zv)^2 \psi \leq \frac{16 \alpha^2}{\ve^2} R^{\ve}\int_{B_R} \rho^{-Q} (Zv)^2  \psi.
\end{align*}
Combining this estimate with \eqref{rt4}, we thus obtain
\begin{equation}\label{rt5}
\alpha^2 \int \rho^{-2\alpha-4+\ve} u^2 \psi \le \frac{4 \alpha^2}{\beta^2} \frac{R^{\ve}}{\ve^2} \int_{B_R} \rho^{-2\alpha}  (\Ba u)^2 \psi^{-1}
\end{equation}

Finally, we show how to incorporate  the integral  $\int_{B_R} \rho^{-2\alpha-2+\ve} |Xu|^2 $ in the left-hand side of \eqref{est1} by a standard interpolation argument. We first observe that
\[
|Xu|^2 = \frac 12 \Ba(u^2) - u \Ba u,
\]
and therefore
\begin{equation}\label{Xu}
\int \rho^{-2\alpha-2+\ve} |Xu|^2 = \frac 12 \int \rho^{-2\alpha-2+\ve} \Ba(u^2) - \int \rho^{-2\alpha-2+\ve}  u \Ba u.
\end{equation}
Using (iv) in Proposition \ref{Est} we integrate by parts in the first term in the right-hand side, obtaining
\begin{align*}
& \int \rho^{-2\alpha-2+\ve} \Ba(u^2) = - \int <X(\rho^{-2\alpha-2+\ve}),X(u^2)>
\\
& = (2\alpha+2-\ve) \int \rho^{-2\alpha-3+\ve} <X\rho,X(u^2)> = (2\alpha+2-\ve) \int \rho^{-2\alpha-4+\ve} Z(u^2) \psi,
 \end{align*}
 where in the last equality we have used \eqref{h10}. Recalling that we have set $u = \rho^\beta v$, we thus find
 \begin{align*}
 &  \int \rho^{-2\alpha-2+\ve} \Ba(u^2) = (2\alpha+2-\ve) \int \rho^{-2\alpha-4 + 2\beta+\ve} Z(v^2) \psi + 2\beta (2\alpha+2-\ve) \int \rho^{-2\alpha-4 +2\beta+\ve} v^2 \psi
 \\
 & = 2 (2\alpha+2-\ve) \int \rho^{-Q+\ve} v  Zv \psi + 2\beta(2\alpha+2-\ve) \int \rho^{-Q+\ve} v^2 \psi
 \\
 & \le  2 (2\alpha+2-\ve) \left(\int \rho^{-Q+2\ve} v^2 \psi\right)^{1/2} \left(\int \rho^{-Q} (Zv)^2 \psi\right)^{1/2}  + 2\beta(2\alpha+2-\ve) \int \rho^{-Q+\ve} v^2 \psi.
 \end{align*}
Using \eqref{h1} in Lemma \ref{hardy} and \eqref{rt4} we can thus bound
\begin{align*}
& \frac 12 \int \rho^{-2\alpha-2+\ve} \Ba(u^2) \le (2\alpha+2-\ve)(16 \beta + 2\ve)\frac{R^\ve}{\ve} \int \rho^{-Q} (Zv)^2 \psi
 \\
 & \le \frac{(2\alpha+2-\ve)(16 \beta + 2\ve)}{4\beta^2}\frac{R^\ve}{\ve}\int_{B_R} \rho^{-2\alpha}  (\Ba u)^2 \psi^{-1} 
\end{align*}
Finally, we have
\begin{align*}
& \left|- \int \rho^{-2\alpha-2+\ve}  u \Ba u\right| \le \left(\int_{B_R} \rho^{-2\alpha}  (\Ba u)^2 \psi^{-1}\right)^{1/2} \left(\int \rho^{-2\alpha-4 +2\ve}  u^2 \psi\right)^{1/2}.
\end{align*} 
Now, using again \eqref{h1} in Lemma \ref{hardy} and \eqref{rt4} we find
\begin{align*}
& \int \rho^{-2\alpha-4 +2\ve}  u^2 \psi = \int \rho^{-Q +2\ve}  v^2 \psi \le \frac{R^{2\ve}}{\beta^2 \ve^2}\int_{B_R} \rho^{-2\alpha}  (\Ba u)^2 \psi^{-1}.
\end{align*}
Substituting in the latter inequality, we conclude
\begin{align*}
& \left|- \int \rho^{-2\alpha-2+\ve}  u \Ba u\right| \le \frac{R^{\ve}}{\beta \ve}\int_{B_R} \rho^{-2\alpha}  (\Ba u)^2 \psi^{-1}.
\end{align*}
Inserting the relevant estimates in \eqref{Xu} we finally have
\begin{equation}\label{Xu2}
\int \rho^{-2\alpha-2+\ve} |Xu|^2 \le \left\{\frac{(2\alpha+2-\ve)(16 \beta + 2\ve) + 4\beta}{4\beta^2}\right\} \frac{R^\ve}{\ve}\int_{B_R} \rho^{-2\alpha}  (\Ba u)^2 \psi^{-1}. 
\end{equation}
The desired conclusion \eqref{est1} now follows from \eqref{rt5} and  \eqref{Xu2}.


\section{Quantitative uniqueness}\label{S:qu}

In this section we prove Theorem \ref{thm3}. Before we proceed with the proof, we establish a Caccioppoli  type inequality which  constitutes one of its essential ingredients.

\begin{lemma}\label{energy}
Let $u$   be a solution to \eqref{e0}, with $V$ satisfying \eqref{vasump}. For any $R>0$ and $0<a<1$, there exists a universal $C = C(a)>0$, such that 
\begin{equation}\label{en}
\int_{B_{(1-a)R} }|Xu|^2 \leq   \frac{C}{a^2R^2} \int_{B_R} ( 1+ K) u^2  \psi.
\end{equation}
\end{lemma}

\begin{proof}
Let $f:\R \to \R$ be a smooth cut-off such that $f(\sigma) \equiv 1$ for $|\sigma|\leq (1-a)R$, $f(\sigma)\equiv 0$ for $|\sigma| \geq R$,
$|f'(\sigma)| \leq \frac{C}{aR}$, and consider
the test function $\phi= f(\rho)^2  u$ in the weak formulation of \eqref{e0}. Using the hypothesis \eqref{vasump}, we obtain from standard computations
\begin{align}\label{e7}
& \int |Xu|^2 f(\rho)^2  \leq  2 \int |u| |Xu| |f(\rho)| f'(\rho) |X\rho| + K f(\rho)^2 u^2\psi
\end{align}
Keeping \eqref{psi} in mind, the Cauchy-Schwarz  inequality gives 
\begin{equation}\label{e8}
2 \int |u| |Xu| |f(\rho)| f'(\rho) |X\rho| \leq \frac{1}{2} \int |Xu|^2 f(\rho)^2 + C \int f'(\rho)^2 u^2 \psi.
\end{equation}
Subtracting the first integral in the right-hand side of \eqref{e8} from the left-hand side in \eqref{e7}, using the bound on $f'$ as the fact that $f(\rho) \equiv 1$ in $B_{(1-a)R}$, the desired conclusion follows.

\end{proof}

We now establish the estimate \eqref{est} in Theorem \ref{thm3}. Such inequality is needed to prove the quantitative uniqueness result in Theorem \ref{main}. We note that \eqref{rt5} above only allows  the following bound
\[
O(\alpha^2) \int \rho^{-2\alpha -4 + \ve} u^2 \psi \leq  \int \rho^{-2\alpha} (\Ba u)^2 \psi^{-1}.
\]
This is why, similarly to the Euclidean case, we are forced to work with the modified  weights containing an additional exponential term. This precisely accounts for  the $O(\alpha^3)$ factor  in front of  the integral $\int \rho^{-2\alpha- 4+\ve} u^2 e^{2\alpha \rho^{\ve}} \psi$ in the estimate  \eqref{est} above.

\begin{proof}[Proof of Theorem \ref{thm3}]

The proof is divided into two steps:

\medskip

\noindent \emph{Step 1}. We first show that
\begin{align}\label{est4}
&\alpha^3\int \rho^{-2\alpha- 4+\ve} u^2 e^{2\alpha \rho^{\ve}} \psi + \alpha^2 \int \rho^{-2\alpha -4} u^2 e^{2\alpha \rho^{\ve}} \psi  \leq \tilde C\int \rho^{-2\alpha}  e^{2\alpha \rho^{\ve}} (\Ba u)^2 \psi^{-1}.
\end{align}
for some $\tilde C$ universal depending also on $\ve, a$. Without restriction, we assume that $u$ be smooth, and let  $v=\rho^{-\beta} e^{\alpha \rho^{\ve}}u$, where as before 
\begin{equation}\label{choice}
\beta = \frac{2\alpha+4 - Q}2
\end{equation}
or equivalently $2\beta - 2 \alpha - 4 = - Q$. With such choice we have $$u=\rho^{\beta} e^{-\alpha \rho^{\ve}} v.$$ This gives 
$$\Ba u=v\Ba(\rho^{\beta}  e^{-\alpha \rho^{\ve}})+2 <X(\rho^{\beta} e^{-\alpha \rho^{\ve}}),X v> + \rho^{\beta} e^{-\alpha \rho^{\ve}} \Ba v.$$
By a standard  calculation we obtain 
\begin{align*}
&\Ba(\rho^{\beta} \ e^{-\alpha \rho^{\ve}})\\
&= \left(\alpha^2 \ve^2 \rho^{\beta+2\ve-2} + \beta(\beta+Q-2)\rho^{\beta-2}- \alpha\ve\left((2\beta+\ve+ Q-2)\right) \rho^{\beta+\ve-2}\right) e^{-\alpha \rho^{\ve}} \psi,   
\notag
\end{align*}
and similarly we have 
\begin{align*}
&2 <X( \rho^{\beta} e^{-\alpha \rho^{\ve}}),X v>= \left(2\beta \rho^{\beta-2} -2 \ve \alpha \rho^{\beta+\ve-2}\right) e^{-\alpha \rho^{\ve}} Zv\ \psi.
\end{align*}
We infer
\begin{align}\label{p1}
& e^{\alpha \rho^{\ve}} \Ba u= 2\beta \rho^{\beta-2} Zv\ \psi +  \rho^{\beta}   \Ba v - 2 \ve \alpha \rho^{\beta+\ve-2} Zv\ \psi 
\\
& + \left(\alpha^2 \ve^2 \rho^{\beta+2\ve-2} + \beta(\beta+Q-2)\rho^{\beta-2}- \alpha\ve\left((2\beta+\ve+ Q-2)\right) \rho^{\beta+\ve-2}\right)  v \psi.
\notag
\end{align}
Using the trivial inequality $(a+b)^2\geq a^2+2ab$, with $a= 2\beta \rho^{\beta-2} Zv\ \psi$, and $b$ given by the remaining terms in the right-hand side of  \eqref{p1}, we obtain
\begin{align}\label{c0}
&\int \rho^{-2\alpha} e^{2\alpha \rho^{\ve}}(\Ba u)^2  \psi^{-1} \geq 4 \beta^2 \int  \rho^{2\beta-2\alpha-4} (Zv)^2 \psi +4 \beta \int \rho^{2\beta-2\alpha-2} Zv   \Ba v
\\
& 
- 8 \alpha \beta \ve \int \rho^{2\beta-2\alpha-4+\ve} (Zv)^2 \psi + 2 \beta^2 (\beta+Q-2) \int \rho^{2\beta-2\alpha-4} Z(v^2)  \psi 
\notag  \\
&  - 2 \alpha \beta \ve(2\beta + \ve + Q -2)\int \rho^{2\beta-2\alpha-4+\ve} Z(v^2) \psi  + 2\alpha^2 \beta \ve^2 \int  \rho^{2\beta-2\alpha-4+2\ve}Z(v^2) \psi.
\notag
\end{align}
Now, from \eqref{zero} above we find
\[
\int \rho^{2\beta-2\alpha-4} Z(v^2) \psi  = \int \rho^{-Q} Z(v^2) \psi  = 0,
\]
whereas \eqref{rt3} gives
\[
4 \beta \int \rho^{2\beta-2\alpha-2} Zv   \Ba v = 4 \beta \int \rho^{-Q+2} Zv   \Ba v \ge 0.
\]
Furthermore, an integration by parts, combined with \eqref{hg} and \eqref{Zpsi}, gives
\[
\int \rho^{2\beta-2\alpha-4+\ve} Z(v^2) \psi = -\ve \int \rho^{-Q+\ve} v^2 \psi,  
\]
\[
\int \rho^{2\beta-2\alpha-4+2\ve} Z(v^2) \psi = -2\ve \int \rho^{-Q+2\ve} v^2 \psi.  
\]
Inserting all of the above in \eqref{c0}, we obtain
\begin{align}\label{c00}
&\int \rho^{-2\alpha} e^{2\alpha \rho^{\ve}}(\Ba u)^2  \psi^{-1} \geq 4 \beta^2 \int  \rho^{-Q} (Zv)^2 \psi - 8 \alpha \beta \ve \int \rho^{-Q+\ve} (Zv)^2 \psi
\\
& + 2 \alpha \beta \ve^2(2\beta + \ve + Q -2) \int \rho^{-Q+\ve} v^2 \psi - 4 \alpha^2 \beta \ve^3 \int \rho^{-Q+2\ve} v^2 \psi.
\notag
\end{align}
Now  since $2\beta - 2 \alpha - 4 = - Q$, therefore    there exists universal constants $K_1$ and $K_2$ depending only on  $Q$  such that for $\alpha$ sufficiently large
\begin{equation}\label{sim}
K_1 \beta  \leq \alpha \leq K_2 \beta.
\end{equation}
This implies, in particular,
\[
8 \alpha \beta \ve \int \rho^{-Q+\ve} (Zv)^2 \psi \leq  8 K_2 \beta^2  \ve R^{\ve} \int \rho^{-Q} (Zv)^2.
\]
Now, if $R$ is chosen small enough, then the latter inequality implies the following estimate
\begin{equation}\label{sma2}
 \beta^2 \int  \rho^{-Q} (Zv)^2 \psi \geq  8 \alpha \beta \ve \int \rho^{-Q+\ve} (Zv)^2 \psi. 
 \end{equation}
Similarly, using \eqref{sim} again, we can  ensure that for a possibly smaller $R$ if needed, one has
\begin{equation}\label{sma4}
 4 \alpha^2 \beta \ve^3 \int \rho^{-Q+2\ve} v^2 \psi \leq   \alpha \beta \ve^2(2\beta + \ve + Q -2) \int \rho^{-Q+\ve} v^2 \psi.
 \end{equation} 
Using \eqref{sma2} and \eqref{sma4} in \eqref{c00}, we thus obtain
\begin{align}\label{c01}
& \int \rho^{-2\alpha} e^{2\alpha \rho^{\ve}}(\Ba u)^2  \psi^{-1} \geq 3 \beta^2 \int  \rho^{-Q} (Zv)^2 \psi + \alpha \beta \ve^2(2\beta + \ve + Q -2) \int \rho^{-Q+\ve} v^2 \psi.
 \end{align}
Finally,  from the Hardy inequality in  Lemma \ref{hardy2} we  have
\begin{equation}\label{c2}
3 \beta^2 \int \rho^{-Q} |Zv|^2 \psi \geq  C\beta^2 \int \rho^{-Q} \psi v^2, 
\end{equation}
where $C>0$ additionally depends on $a$. The reader should note that this  is precisely the place where we need that  $u$, and consequently $v$, is supported in $B_{R} \setminus B_{aR}$. 
The proof of the desired estimate \eqref{est4} is then completed using  \eqref{c2}  and  \eqref{sim} in \eqref{c01}.

\medskip

\noindent \emph{Step 2.}
Now we show that  there exists $C_1>0$ such that, if with $K$ as in \eqref{vasump} one has 
\begin{align}\label{alpha}
\alpha  \geq C_1 (K^{2/3} +1),
\end{align}
then the following estimate holds
\begin{align}\label{est5}
&\alpha^3\int \rho^{-2\alpha- 4+\ve} u^2 e^{2\alpha \rho^{\ve}}  \psi  + \alpha^2 \int \rho^{-2\alpha -4} u^2 e^{2\alpha \rho^{\ve}} \psi  \leq C \int \rho^{-2\alpha}  e^{2\alpha \rho^{\ve}}  (\Ba u +V u)^2 \psi^{-1}.
\end{align}
for some universal $C>0$ depending also on $a, \ve$. This follows from \eqref{est4} in a straightforward way by noting that from the growth assumption on $V$ in  \eqref{vasump}, using the simple inequality $(a+b)^2 \geq \frac{1}{2}a^2 - 4 b^2$, we obtain 
\begin{align*}
& \int \rho^{-2\alpha}  e^{2\alpha \rho^{\ve}}  (\Ba u +V u)^2 \psi^{-1} 
 \geq \frac{1}{2} \int \rho^{-2\alpha}  e^{2\alpha \rho^{\ve}}  (\Ba u)^2 \psi^{-1} - 4 K^2 \int \rho^{-2 \alpha} e^{2\alpha \rho^{\ve}} u^2 \psi.
\notag
\end{align*}
If now $C_1$ in \eqref{alpha} is chosen large enough in dependence of $\tilde C$ in \eqref{est4}, then we can ensure that   
\[
\alpha^3 > 16 \tilde C K^2
\]
holds. Consequently, we can infer from \eqref{est4} that the following inequality holds
\begin{align*}
4 K^2 \int \rho^{-2 \alpha} e^{2\alpha \rho^{\ve}} u^2 \psi \leq \frac{1}{4} \int \rho^{-2\alpha}  e^{2\alpha \rho^{\ve}} (\Ba u)^2 \psi^{-1}. 
\end{align*}
This implies that 
\[
 \int \rho^{-2\alpha} e^{2\alpha \rho^{\ve}} (\Ba u +V u)^2 \psi^{-1} \geq \frac{1}{4} \int \rho^{-2\alpha}  e^{2\alpha \rho^{\ve}} (\Ba u)^2 \psi^{-1},
\]
from  which the desired estimate \eqref{est5} follows by applying the estimate \eqref{est4} in Step 1.  This completes the proof of the theorem.

\end{proof}

\begin{proof}[Proof of Theorem \ref{main}]
We adapt an Euclidean argument in \cite{Bk1}. For a given $R_1 < R_2$, $A_{R_1, R_2}$ will denote the annulus $B_{R_1} \setminus B_{R_2}$. Let $R_0$ be as in Theorem \ref{thm3} and let $ 0 < R< \frac{R_0}{2}$. Also let  $\phi \in C_{0}^{\infty}(B_{2R})$  such that
\begin{equation}\label{bd1}
\begin{cases}
\phi \equiv 0\ \text{if $\rho < \frac{R}{4}$ and $\rho > \frac{5R}{3}$}
\\
\phi \equiv 1\ \text{in $A_{\frac{R}{3}, \frac{3R}{2}}$.}
\end{cases}
\end{equation}
As in the proof of the energy estimate  in Lemma \ref{energy},  we can take $\phi$  to be  a radial function of the form $f(\rho)$, and therefore we can ensure that  the following bounds hold, 
\begin{equation}\label{bd}
|X \phi| \leq  \frac{C\psi^{1/2}}{R},\ \ \ \ \ \ \ \ | \Ba \phi| \leq \frac{C \psi}{R^2}.
\end{equation}

Using the equation \eqref{e0} satisfied by $u$, from the estimate \eqref{est} applied to $u \phi$, we thus obtain for some universal constant $C>0$ depending on $m,n$ and  $\gamma$, 
\begin{equation}\label{ty11}
\alpha^2 \int  \rho^{-2\alpha - 4} e^{2\alpha \rho^{\ve}} u^2 \phi^2 \psi \leq C \int \rho^{-2\alpha} e^{2 \alpha \rho^{\ve}} (u^2 (\Ba \phi)^2 \psi^{-1}+ |Xu|^2 |X \phi|^2 \psi^{-1}).
\end{equation}
In what follows we respectively indicate with $||f||_R$ and $||f||_{R_1, R_2}$ the $L^{2}$ norm of $f$ in $B_R$ and $A_{R_1, R_2}$. 
From \eqref{bd1}, \eqref{bd}, and the fact that  the functions $X\phi, \Ba \phi$ are supported in $\{ (z,t) \mid \frac{R}{4} < \rho(z,t) < \frac{R}{3}\} \cup \{(z,t) \mid \frac{3R}{2} < \rho(z,t)< \frac{5R}{3}\}$, we obtain from \eqref{ty11}   
\begin{align}\label{bd2}
& || \rho^{-\alpha -2} e^{\alpha \rho^{\ve}} u \psi^{1/2}||_{\frac{R}{3}, \frac{3R}{2}} \leq C\left ( || \rho^{-\alpha -2} e^{\alpha \rho^{\ve}} u \psi^{1/2}||_{\frac{R}{4}, \frac{R}{3}} +  || \rho^{-\alpha -2} e^{\alpha \rho^{\ve}} u \psi^{1/2}||_{\frac{3R}{2}, \frac{5R}{3}} \right)
\\
& + C R \left( || \rho^{-\alpha -2} e^{\alpha \rho^{\ve}} |Xu| ||_{\frac{R}{4}, \frac{R}{3}} +  || \rho^{-\alpha -2} e^{\alpha \rho^{\ve}} |Xu| ||_{\frac{3R}{2}, \frac{5R}{3}} \right).
\notag
\end{align}
Bounding from below the integral in the left-hand side of \eqref{bd2} with one on $A_{\frac{R}{3},R}$, we find for some universal $C>0$,
\begin{align}\label{b2}
& || \rho^{-\alpha -2} e^{\alpha \rho^{\ve}} u \psi^{1/2}||_{\frac{R}{3}, R} \leq C\left ( || \rho^{-\alpha -2} e^{\alpha \rho^{\ve}} u \psi^{1/2}||_{\frac{R}{4}, \frac{R}{3}} +  || \rho^{-\alpha -2} e^{\alpha \rho^{\ve}} u \psi^{1/2}||_{\frac{3R}{2}, \frac{5R}{3}} \right)
\\
& + C R \left( || \rho^{-\alpha -2} e^{\alpha \rho^{\ve}} |Xu| ||_{\frac{R}{4}, \frac{R}{3}} +  || \rho^{-\alpha -2} e^{\alpha \rho^{\ve}} |Xu| ||_{\frac{3R}{2}, \frac{5R}{3}} \right).
\notag
\end{align} 
We now consider the function $h(r) = r^{-\alpha-2} e^{\alpha r^{\ve}}$. One easily checks that, for $R_0$ sufficiently small, the function $r\to h(r)$ is decreasing  in $(0,R_0)$, and moreover the following  inequalities hold for universal $C_4>0, C_5>1$, 
\begin{equation}\label{bd10}
\frac{h(r/4)}{h(r)} \leq C_4^{\alpha},\ \ \ \ \ \ \ 
\frac{h(3r/2)}{h(r) } \leq C_5^{-\alpha}.
\end{equation}
Therefore, possibly taking a smaller $R_0$ in the statement of Theorem \ref{thm3},  using \eqref{bd10} for $R \leq R_0$, we obtain from \eqref{b2} the following unweighted estimate 
\begin{align}\label{bd4}
& ||  u \psi^{1/2}||_{R/3, R} \leq  C C_4^{\alpha} \left ( || u \psi^{1/2}||_{\frac{R}{4}, \frac{R}{3}} + R ||  |Xu| ||_{\frac{R}{4}, \frac{R}{3}}   \right)
\\
& +  C C_5^{-\alpha}  \left(||  u \psi^{1/2}||_{\frac{3R}{2}, \frac{5R}{3}} + R  || |Xu| ||_{\frac{3R}{2}, \frac{5R}{3}} \right).
\notag
\end{align}
Now, the Caccioppoli estimate in Lemma \ref{energy} gives
\begin{equation}\label{bd5}
\begin{cases}
R  ||  |Xu| ||_{\frac{R}{4}, \frac{R}{3}} \leq C(1+K^{1/2}) ||u\psi^{1/2}||_{\frac{R}{2}},
\\
 R|| |Xu| ||_{\frac{3R}{2}, \frac{5R}{3}}   \leq C( 1+ K^{1/2})  || u \psi^{1/2}||_{2R}.
 \end{cases}
\end{equation}
Adding $||u \psi^{1/2}||_{\frac{R}{3}}$ to both sides of  \eqref{bd4}, and using \eqref{bd5},  we obtain for some different constants $C, C_4, C_5$,  
\begin{equation}\label{bd7}
||u\psi^{1/2}||_{R} \leq C( 1+ K^{1/2}) ( C_4^{\alpha} ||u \psi^{1/2}||_{\frac{R}{2}} + C_5^{-\alpha} ||u \psi^{1/2}||_{2R}).
\end{equation}
At this point, similarly to \cite{Bk1}, we choose $\alpha$ sufficiently large such that   the following holds
\[
C( 1+ K^{1/2}) C_5^{-\alpha}   ||u \psi^{1/2}||_{2R} \leq \frac{1}{2} ||u \psi^{1/2}||_R.
\]
(The reader should note here that the  variable $\tau$ in \cite{Bk1} plays the  same role as  our $\alpha$). A standard real analysis argument as on p. 78-79  in \cite{Bk1}, now shows the existence of universal $K_2,  M>0$ (depending also on $K_0$ in Theorem \ref{main}) such that the following inequality holds, 
\begin{equation}\label{f1}
||u \psi^{1/2}||_{R}^{1+K_2} e^{- M(K^{2/3} +1 )}  \leq ||u \psi^{1/2}||_{\frac{R}{2}}.
\end{equation}
Note that the estimate \eqref{f1} is analogous to (4.10) in \cite{BG1}, except that the power $K^{2/3}$ replaces $K^{1/2}$. We can thus repeat the standard  iterative  argument in \cite[Sec. 4]{BG1} to conclude the validity of  the estimate \eqref{main1} in Theorem \ref{main}.

\end{proof}

\section{Strong unique continuation for sublinear equations}\label{S:sub}

In this final section we prove Theorem \ref{main4}. Similarly to Theorem 2 in \cite{Ru},  we accomplish this by means of the following Carleman estimate.

\begin{thrm}\label{sub1}
Let $1< q< 2$, and let $f$ satisfy the assumptions \eqref{a1}. For  every $\ve>0$, there exists $C=C(m, k, \ve, \la, q, K, \kappa, c_0)>0$ such that for $\alpha >0$ sufficiently large, and $u \in S^{2,2}_{0}(B_R \setminus \{0\})$ with  $\text{supp}\ u \subset (B_R \setminus \{0\})$, one has
\begin{align}\label{f10}
\alpha^2 \int \rho^{-2\alpha-4+\ve}u^2 \psi+  \rho^{-2\alpha-2+\ve} |u|^q \psi^2   \leq C R^{\ve}  \int \rho^{-2\alpha}(\Ba u + f((z,t), u) \psi)^2 \psi^{-1}.
\end{align}\end{thrm}

\begin{proof}
The proof is similar to that of Theorem \ref{thm2} except that we additionally exploit the specific  nature of the sublinearity $f((z,t), u)$.  As before, we let $u= \rho^{\beta}v$ where $\alpha$ and $\beta$ related by \eqref{choice}.  In terms of $v$, we  have
\begin{align}\label{f11}
\Ba u + f((z,t), u) \psi &= \beta(\beta+Q-2)\rho^{\beta-2} v \psi +2\beta \rho^{\beta-2} Zv \psi+ \Ba v \rho^{\beta} + f((z,t), \rho^\beta v)  \psi.
\end{align}
The integral 
\[
\int \rho^{-2\alpha}(\Ba u + f((z,t), u) \psi)^2 \psi^{-1}
\]
is now estimated from below using again the inequality $(a+b)^2 \geq a^2 + 2ab$, with $a= 2\beta \rho^{\beta-2} Zv \psi$ and $b$ being the rest of the terms in \eqref{f11}. We note that all the other terms are handled precisely as in the proof of Theorem \ref{thm2}, with the exception of    
\begin{equation*}
4 \beta \int \rho^{-2\alpha + \beta -2} f((z,t),\rho^\beta v) Zv \psi^2.
\end{equation*}
This integral occurs for the presence of the sublinear term in \eqref{sub}, and our next goal is to bound it from below. With this in mind, we observe that since $F$ is the antiderivative of $f$ in the $s$ variable, we have
\begin{align*}
& Z (F((z,t), \rho^{\beta}v))= f((z,t),\rho^\beta v)  \rho^{\beta} Zv + \beta f((z,t),\rho^{\beta} v) \rho^{\beta} v +
 <\nabla_{(z,t)} F((z,t),\rho^{\beta} v),Z>.
\end{align*}
This identity gives
\begin{align}\label{m10}
&4 \beta \int \rho^{-2\alpha + \beta -2}    f((z,t), \rho^\beta v) Zv \psi^2= 4 \beta \int\rho^{-2\alpha-2} Z (F((z,t),\rho^{\beta}v)) \psi^2
\\
& - 4\beta^2 \int \rho^{-2\alpha -2 } f((z,t), \rho^{\beta}v) \rho^{\beta} v \psi^2  -4 \beta \int \rho^{-2\alpha-2} < \nabla_{(z,t)} F((z,t), \rho^{\beta}v), Z>  \psi^2.
\notag
\end{align}
From the third condition in \eqref{a1} one has 
\begin{equation}\label{i8}
f((z,t),\rho^{\beta}v) \rho^{\beta}v \leq q F((z,t),\rho^{\beta} v),\ \ \ \ \ \ \ \big<\nabla_{(z, t)} F, Z\big> \leq  C_2 F,
\end{equation}
where $C_2$ depends on $m,k, \gamma$, and on $K$ in \eqref{a1}. Using \eqref{i8} in \eqref{m10}, we obtain 
\begin{align}\label{m11}
&4 \beta  \int \rho^{-2\alpha + \beta -2}   f((z,t), \rho^\beta v) Zv \psi^2  \geq 4 \beta \int Z(F((z,t), \rho^{\beta}v))\rho^{-2\alpha-2} \psi^2
\\
& - 4\beta^2 q  \int \rho^{-2\alpha -2} F((z,t), \rho^\beta v) \psi^2- 4 C_2 \beta \int \rho^{-2\alpha-2} F((z,t), \rho^{\beta} v)  \psi^2.
\notag
\end{align}
Now, the first term in the right-hand side of \eqref{m11}
is handled using integration by parts as follows 
\begin{align}\label{m13}
&4 \beta \int Z \big( F((z,t), \rho^{\beta}v) \big)\rho^{-2\alpha-2} \psi^2 = - 4 \beta \int F((z,t), \rho^{\beta} v) \operatorname{div}(\rho^{-2\alpha-2} Z) \psi^2
\\
& = 8\beta (\beta -1) \int \rho^{-2\alpha-2} F((z,t), \rho^{\beta} v)  \psi^2
\notag
\end{align}
We note that in the last equality in \eqref{m13} we have used \eqref{Zpsi}, which gives $Z\psi^2 \equiv 0$, and (iii) in Proposition \ref{Est} along with \eqref{choice}, to assert that
\[
\operatorname{div}(\rho^{-2\alpha-2} Z)= -2(\beta-1)\rho^{-2\alpha-2}.
\]
Using \eqref{a0} and \eqref{m13}  in \eqref{m11}, we conclude
\begin{align*}
 &4 \beta \int \rho^{-2\alpha + \beta -2} f((z,t), \rho^\beta v) Zv  \psi^2  \geq 4 \beta^2((2-q) - (C_2 +2)\beta^{-1})   \int \rho^{-2\alpha-2} F((z,t),\rho^{\beta} v) \psi^2.
 \end{align*}
Keeping in mind that $2-q>0$, if we choose $\beta>\frac{2(C_2 +2)}{2-q}$, then $(2-q) - (C_2 +2)\beta^{-1} > \frac{2-q}2$, and using the bound from below for $F$ in \eqref{a0}, and that from above in \eqref{sim}, we find
 \begin{align*}
 & 4 \beta \int \rho^{-2\alpha + \beta -2} f((z,t), \rho^\beta v) Zv  \psi^2  \ge C \alpha^2 \int \rho^{-2\alpha-2} |u|^q \psi^2,
 \end{align*}
for $\alpha$ sufficiently large, depending on $q, C_2$, and for some universal $C>0$ depending on $q$ and  $c_0$. 
This estimate accounts for the second term in the left-hand side of \eqref{f10}. Since, as we have said above, the first term is obtained exactly as in the proof of Theorem \ref{thm2}, this completes the proof of \eqref{f10}. 

\end{proof}

With Theorem \ref{sub1} in hand, we now proceed to proving the main result in this section. 

\begin{proof}[ Proof of Theorem \ref{main4}]
With $u$ as in the hypothesis of Theorem \ref{main4}, we take $u_\ve= \phi_\ve u$, 
where $\phi_\ve$ is a smooth cutoff such that
\begin{equation}
\begin{cases}
\phi_\ve(z,t) \equiv 1,\ \text{when}\ \ve < \rho(z,t) \leq \frac{1}{2},
\\
\phi_\ve \equiv 0, \ \text{when}\ \rho(z,t) \leq \frac{\ve}{2}, \text{or}\ \rho(z,t) > 1.
\end{cases}
\end{equation}
Using the equation \eqref{sub} satisfied by $u$ we thus find 
\[
\Ba u_\ve + f((z,t), u_\ve) \psi =   2 <Xu, X\phi_\ve> + u \Ba \phi_\ve +   \psi \left(f((z,t), u_\ve) - f((z,t), u) \phi_\ve \right) - Vu_{\ve}.
\]
Applying the Carleman estimate \eqref{f10} with $R=1$, and with $u_\ve$ instead of $u$, using the bound on $V$ in \eqref{vasump} we obtain,
\begin{align*}
& \alpha^2 \int \big(\rho^{-2\alpha-4+\delta}u_\ve^2 \psi+ \rho^{-2\alpha-2} |u_\ve|^q \psi^2 \big) 
\\
& \leq C  \int \rho^{-2\alpha} \big( |Xu|^2 |X\phi_\ve|^2 \psi^{-1}  + u^2 |\Ba \phi_\ve|^2 \psi^{-1}  + (f((z,t), u_\ve) - f((z,t), u) \phi_\ve)^2 \psi + K^2 u_\ve^2 \psi\big).
\end{align*}
We now observe that if $\alpha$ is large enough, depending on $K$, then the integral contianing the last term in the right-hand side of this inequality can be absorbed in the left-hand side, obtaining
\begin{align}\label{f13}
& \alpha^2 \int \big(\rho^{-2\alpha-4+\delta}u_\ve^2 \psi+ \rho^{-2\alpha-2} |u_\ve|^q \psi^2 \big) 
\\
& \leq C  \int \rho^{-2\alpha} \big( |Xu|^2 |X\phi_\ve|^2 \psi^{-1}  + u^2 |\Ba \phi_\ve|^2 \psi^{-1}  + (f((z,t), u_\ve) - f((z,t), u) \phi_\ve)^2 \psi\big).
\notag
\end{align}
Next, we observe that the following energy estimate holds for any $0<r<1/2$, for some universal $C>0$,
\[
\int_{B_r } |Xu|^2 \leq  \frac{C}{r^2} \int_{B_{2r} } ( u^2 + |u|^p) \psi
\]
Its proof is completely similar to that of Lemma \ref{energy}, if one uses the bounds in \eqref{vasump}, \eqref{a1} and \eqref{a0}.
From this energy estimate it is immediate to see that $|Xu|$ also  vanishes to infinite order at the origin, i.e., as $r\to 0^+$, for any $\ell\in \mathbb N$ one has 
\begin{equation}\label{vx}
\int_{B_r} |Xu|^2 = O(r^\ell). 
\end{equation}
Noting that the functions  $X\phi_\ve, \Ba \phi_\ve$ are supported in $\{\frac{\ve}{2} < \rho(z, t) < \ve\} \cup \{ \frac{1}{2} < \rho(z,t) < 1\}$, and for some universal $C>0$ satisfy the bounds  
\begin{equation}\label{a2}
\begin{cases}
|X\phi_\ve| \leq \frac{C \psi^{1/2}}{\ve}, \ |\Ba \phi_\ve| \leq \frac{C \psi}{\ve^2}, \ \text{when} \ \frac{\ve}{2} < \rho(z,t) < \ve,
\\
|X\phi_\ve| \leq C \psi^{1/2}, \ |\Ba \phi_\ve| \leq C \psi,\ \text{when} \ \frac{1}{2} < \rho(z,t) < 1.
\end{cases}
\end{equation}
from the vanishing to infinite order of $u$ and $Xu$, and from \eqref{a2}, we can assert that as $\ve \to 0$,  
\[
\int_{B_{\ve} \setminus B_{\frac{\ve}{2}}} \rho^{-2\alpha} \big( |Xu|^2 |X\phi_\ve|^2 \psi^{-1}  + u^2 |\Ba \phi_\ve|^2 \psi^{-1 }+ u_{\ve}^2 \psi + (f((z,t), u_\ve) - f((z,t), u) \phi_\ve)^2 \psi \big)  \to 0, 
\]
and
\[
\int_{B_{\ve} \setminus B_{\frac{\ve}{2}}} \rho^{-2\alpha-4+\delta}u_\ve^2 \psi+ \rho^{-2\alpha-2} |u_\ve|^q \psi^2   \to 0.
\]
Therefore, if we denote by $\phi_0$ the pointwise limit of $\phi_\ve$, letting $\ve \to 0$ in \eqref{f13}, we consequently obtain for $u_0= \phi_0 u $ the following inequality
\begin{align}\label{f14}
& \int \left[\rho^{-2\alpha-4+\delta}u_0^2 \psi+ \rho^{-2\alpha-2} |u_0|^q \psi^2 \right]  
\\
& \leq C \, \int \rho^{-2\alpha} ( |Xu|^2 |X\phi_0|^2 \psi^{-1}  + u^2 |\Ba \phi_0|^2 \psi^{-1}  + (f((z,t), u_0) - f((z,t), u) \phi_0)^2  \psi)
\notag
\end{align}
Noting that the integrals in the right-hand side of  \eqref{f14} are supported in the region $\{1/2 < \rho(z, t)  < 1\}$. On the other hand, we have $u_0 \equiv u$ when $\rho(z,t) < \frac{1}{2}$. Consequently, we minorise the integral in the left-hand side of \eqref{f14} if we take $B_{1/4}$ as set of integration. Using in \eqref{a2} the derivative bounds corresponding to $\phi_0$, we finally conclude with the estimate
\begin{align*}
& 4^{2 \alpha +2} \int_{B_{1/4}} u^2 \psi
\leq C  2^{2\alpha}   \int_{B_{1} \setminus B_{1/2}} ( |Xu|^2  + u^2 \psi  + (f((z,t), u_0) - f((z,t), u) \phi_0)^2  \psi).
\end{align*}
Letting $\alpha \to  \infty$ in this inequality, we infer $u \equiv 0$ in $B_{1/4}$. Since the operator $\Ba$ is translation invariant in the variable $t$,  by a standard  covering argument we obtain $u \equiv 0$ for  $|z| < \frac{1}{4}$. Since in the region $|z|>\frac{1}{4}$ the operator $\Ba$ is uniformly elliptic with Lipschitz coefficients, applying  the results from \cite{Ru} and \cite{SW}, we can finally deduce  that $u \equiv 0$ in $B_1$. 

\end{proof}

In closing, we remark that is easy to see  that the  Carleman estimate \eqref{f10} also implies the strong unique continuation property for sublinear equations when the  potential $V$ satisfies the  following growth condition
\[
|V(z,t)| \leq \frac{C \psi}{\rho^{2-\delta}}
\]
for some $C, \delta >0$.


\begin{thebibliography}{99}


\bibitem{AB}
S. Alinhac \& M. S. Baouendi, \emph{A counterexample to strong uniqueness for partial differential equations of Schr\"odinger's type}. 
Comm. Partial Differential Equations \textbf{19}~(1994), no. 9-10, 1727-1733. 


\bibitem{Al}
F. J. Almgren, Jr., \emph{Dirichlet's problem for multiple valued functions and the regularity of mass minimizing integral currents. Minimal submanifolds and geodesics}, (Proc. Japan-United States Sem., Tokyo, 1977), North-Holland, Amsterdam-New York, 1979.

\bibitem{A}
N. Aronszajn, 
\emph{A unique continuation theorem for solutions of elliptic partial differential equations or inequalities of second order}, J. Math. Pures Appl. (9) \textbf{36}~ (1957), 235-249.

\bibitem{AKS}
N. Aronszajn, A. Krzywicki \& J. Szarski, \emph{A unique continuation theorem for exterior differential forms on Riemannian manifolds},Ark. Mat. \textbf{4}~ 1962 417-453 (1962).

\bibitem{Bao}
H. Bahouri, \emph{Non prolongement unique des solutions d'op\'erateurs ``somme de carr\'es''}. (French) [Failure of unique continuation for "sum of squares'' operators] Ann. Inst. Fourier (Grenoble) \textbf{36}~(1986), no. 4, 137-155. 

\bibitem{Bk}
L. Bakri, \emph{Quantitative uniqueness for Schr\"odinger operator}, Indiana Univ. Math. J., \textbf{61}~(2012), no. 4, 1565-1580. 

\bibitem{Bk1}
L. Bakri, \emph{Carleman estimates for the Schr\"odinger operator. Applications to quantitative uniqueness}. Comm. Partial Differential Equations \textbf{38}~(2013), no. 1, 69-91.




\bibitem{B}
A. Banerjee, \emph{Sharp vanishing order of solutions to stationary Schr\"odinger equations on Carnot groups of arbitrary step},  J. Math. Anal. Appl. \textbf{465}~ (2018), no. 1, 571-587.


\bibitem{BM}
A. Banerjee \& A. Mallick, \emph{On the strong unique continuation of a degenerate elliptic operator with Hardy type potential}, 	arXiv:1807.01947

\bibitem{BG}
A. Banerjee \& N. Garofalo, \emph{Quantitative uniqueness for elliptic equations at the boundary of $C^{1, Dini}$ domains}, J. Differential Equations \textbf{261}~ (2016), no. 12, 6718-6757.

\bibitem{BG1}
A. Banerjee  \& N. Garofalo \emph{Quantitative uniqueness for zero-order perturbations of generalized Baouendi-Grushin operators}, Rend. Istit. Mat. Univ. Trieste \textbf{48}~ (2016), 189-207.

\bibitem{Ba}
S. M. Baouendi, \emph{Sur une classe d'op\'erateurs elliptiques d\'eg\'en\'er\'es} (French) Bull. Soc. Math. France, \textbf{95}~1967, 45-87.

\bibitem{BK}
J. Bourgain \& C. Kenig, \emph{On localization in the continuous Anderson-Bernoulli model in higher dimension}, Invent. Math. \textbf{161}~ (2005), no. 2, 389-426.

\bibitem{CSS}
L. A. Caffarelli, S. Salsa \& L. Silvestre,  \emph{Regularity estimates for the solution and the free boundary of the obstacle problem for the fractional Laplacian}, Invent. Math. \textbf{171}~(2008), no. 2, 425-461.


\bibitem{CS}
L. Caffarelli \& L. Silvestre, \emph{An extension problem related to the fractional Laplacian}, Comm. Partial Differential Equations \textbf{32}~(2007), no.~7--9, 1245--1260.


\bibitem{C}
T. Carleman, \emph{Sur un probl\`eme d'unicit\'e pur les systemes d'\'equations aux d\'eriv\'ees partielles \`a deux variables ind\'ependantes}. (French) Ark. Mat., Astr. Fys. \textbf{26}~(1939). no. 17, 9 pp.



\bibitem{ChS}
S. Chanillo \& E. Sawyer, \emph{Unique continuation for $\Delta + \nu$ and the C. Fefferman-Phong class}, Trans. Amer. Math. Soc. \textbf{318}~ (1990), no. 1, 275-300. 

\bibitem{Co}
H. O. Cordes, \emph{\"Uber die eindeutige Bestimmtheit der L\"osungen elliptischer Differentialgleichungen durch Anfangsvorgaben}. (German) Nachr. Akad. Wiss. G\"ottingen. Math.-Phys. Kl. IIa. 1956 (1956), 239-258. 
 

\bibitem{DF1}
H. Donnelly \& C. Fefferman, \emph{Nodal sets of eigenfunctions on Riemannian manifolds}, Invent. Math, \textbf{93}~ (1988), 161-183.


\bibitem{DF2}
H. Donnelly \& C. Fefferman, \emph{Nodal sets of eigenfunctions: Riemannian manifolds with boundary}, Analysis, Et Cetera, Academic Press, Boston, MA, 1(990,) 251-262.

\bibitem{EV}
L. Escauriaza \& S. Vessella, \emph{Optimal three-cylinder inequalities for solutions to parabolic equations with Lipschitz leading coefficients}. Contemp. Math. \textbf{333}, 79-87.

\bibitem{FL0}
B. Franchi \& E. Lanconelli, \emph{Une m\'etrique associ\'ee \`a une classe d'op\'erateurs elliptiques d\'eg\'en'er'es}. (French) [A metric associated with a class of degenerate elliptic operators] Conference on linear partial and pseudodifferential operators (Torino, 1982). Rend. Sem. Mat. Univ. Politec. Torino 1983, Special Issue, 105-114 (1984).


\bibitem{FL}
B. Franchi \& E. Lanconelli, \emph{H\"older regularity theorem for a class of linear non uniformly elliptic operators with measurable coefficients}, Ann. Sc. Norm. Sup. Pisa \textbf{4}~(1983), 523-541.

\bibitem{FL3}
B. Franchi \& E. Lanconelli, \emph{An embedding theorem for Sobolev spaces related to nonsmooth vector fields and Harnack inequality}. Comm. Partial Differential Equations \textbf{9}~(1984), no. 13, 1237-1264. 


\bibitem{G}
N. Garofalo, \emph{Unique continuation for a class of elliptic operators which degenerate on a manifold of arbitrary codimension.}, J. Diff. Equations \textbf{104}~ (1993), no. 1, 117-146.

\bibitem{Gmem}
N. Garofalo, \emph{Hypoelliptic operators and some aspects of analysis and geometry of sub-Riemannian spaces}, Geometry, Analysis and Dynamics on sub-Riemannian Manifolds, EMS Series of Lectures in Mathematics, Vol. I, Editors:
D. Barilari, U. Boscain, M. Sigalotti, 2016, Springer. 

\bibitem{GLan}
N. Garofalo \& E. Lanconelli, \emph{Frequency functions on the Heisenberg group, the uncertainty principle and unique continuation}. Ann. Inst. Fourier (Grenoble) \textbf{40}~(1990), no. 2, 313-356. 

\bibitem{GL1}
N. Garofalo \& F. Lin, \emph{Monotonicity properties of variational integrals, $A_p$ weights and unique continuation}, Indiana Univ. Math. J. \textbf{35}~(1986),  245-268.

\bibitem{GL2}
N. Garofalo \& F. Lin, \emph{Unique continuation for elliptic  operators: a geometric-variational approach}, Comm. Pure Appl. Math. \textbf{40}~ (1987),  347-366.

\bibitem{GRO}
N. Garofalo \& X. Ros-Oton, \emph{Structure and regularity of the singular set in the obstacle problem for the fractional Laplacian},  Revista Mat. Iberoamer., to appear.


\bibitem{GR}
N. Garofalo \& K. Rotz, \emph{Properties of a frequency of Almgren type for harmonic functions in Carnot groups}. Calc. Var. Partial Differential Equations \textbf{54}~(2015), no. 2, 2197-2238. 

\bibitem{GarShen}
N. Garofalo \& Z. Shen, \emph{Carleman estimates for a subelliptic operator and unique
  continuation}. Ann. Inst. Fourier (Grenoble) \textbf{44}~(1994), no. 1, 129-166. 

\bibitem{GV}
N. Garofalo \& D. Vassilev, \emph{Strong unique continuation properties of generalized Baouendi-Grushin operators.}, Comm. Partial Differential Equations \textbf{32}~ (2007), no. 4-6, 643-663. 

\bibitem{Gra}
C. Grammatico, \emph{A result on strong unique continuation for the Laplace operator}. Comm. Partial Differential Equations \textbf{22}~ (1997), no. 9-10, 1475-1491.

\bibitem{Gr1}
V. V. Grushin, \emph{A certain class of hypoelliptic operators}, (Russian) Mat. Sb. (N.S.) \textbf{83} (125)~(1970), 456-473.

\bibitem{Gr2}
\bysame, \emph{A certain class of elliptic pseudodifferential operators that are degenerate on a submanifold}, (Russian) Mat. Sb. (N.S.) \textbf{84} (126)~(1971), 163-195.

\bibitem{HW}
P. Hartman \& A. Wintner, \emph{On the local behavior of solutions of non-parabolic partial differential equations. III. Approximations by spherical harmonics}. 
Amer. J. Math. \textbf{77}~(1955), 453-474. 

\bibitem{He}
E. Heinz, \emph{\"Uber die Eindeutigkeit beim Cauchyschen Anfangswertproblem einer elliptischen Differentialgleichung zweiter Ordnung}. (German) 
Nachr. Akad. Wiss. G\"ottingen. IIa. 1955 (1955), 1-12. 

\bibitem{Ho}
L. H\"ormander, \emph{Uniqueness theorems for second order elliptic differential equations}. Comm. Partial Differential Equations \textbf{8}~(1983), no. 1, 21-64.

\bibitem{J}
D. Jerison, \emph{Carleman inequalities for the Dirac and Laplace operators and unique continuation}. Adv. in Math. \textbf{62}~ (1986), no. 2, 118-134.

\bibitem{JK}
  D. Jerison \& C. Kenig, \emph{Unique continuation and absence of positive eigenvalues for Schrodinger operators}, Ann. of Math. (2) \textbf{121}~(1985), no. 3, 463-494. 


\bibitem{KPS}
H. Koch, A. Petrosyan \& W. Shi, \emph{Higher regularity of the free boundary in the elliptic Signorini problem}, Nonlinear Anal. \textbf{126}~(2015), 3-44.

\bibitem{KT}
  H. Koch \& D. Tataru, \emph{Carleman estimates and unique continuation for second-order elliptic equations with nonsmooth coefficients}, Comm. Pure Appl. Math. 54 (2001), no. 3, 339-360.  




\bibitem{MV}
E. Malinnikova \& S. Vessella, \emph{Quantitative uniqueness for elliptic equations with singular lower order terms}. Math. Ann. \textbf{353}~(2012), no. 4, 1157-1181. 


\bibitem{Me}
V. Meshov, \emph{On the possible rate of decrease at infinity of the solutions of second-order partial differential equations}, Math. USSR-Sb. \textbf{72}~(1992), no. 2, 343-361.

\bibitem{Mi}
K. Miller, \emph{Nonunique continuation for uniformly parabolic and elliptic equations in self-adjoint divergence form with H\"older continuous coefficients}. Arch. Rational Mech. Anal. \textbf{54}~(1974), 105-117. 

\bibitem{Mu}
C. M\"uller, \emph{On the behavior of the solutions of the differential equation and $\Delta U=F(x,U)$ in the neighborhood of a point}. 
Comm. Pure Appl. Math. \textbf{7}~(1954), 505-515.  

\bibitem{PW}
Y. Pan \& T. Wolff, \emph{
A remark on unique continuation}.  
J. Geom. Anal. \textbf{8}~(1998), no. 4, 599-604. 

\bibitem{Pl}
A. Plis, \emph{On non-uniqueness in Cauchy problem for an elliptic second order differential equation}, Bull. Acad. Polon. Sci. Ser. Sci. Math. Astronom. Phys. \textbf{11}~(1963), 95-100.

\bibitem{Reg}
R. Regbaoui, \emph{Strong uniqueness for second order differential operators}. (English summary) 
J. Differential Equations \textbf{141}~(1997), no. 2, 201-217.

\bibitem{Ru}
A. R\"uland, 
\emph{Unique Continuation for sublinear elliptic equations based on Carleman estimates}. J. Differential Equations \textbf{265}~(2018), no. 11, 6009-6035.

\bibitem{Ru1}
A. R\"uland, \emph{On quantitative unique continuation properties of fractional Schr\"odinger equations: doubling, vanishing order and nodal domain estimates}, Trans. Amer. Math. Soc. \textbf{369}~ (2017), no. 4, 2311-2362.

\bibitem{SW}




N. Soave \& T. Weth, \emph{The unique continuation property of sublinear equations}. SIAM J. Math. Anal. \textbf{50}~(2018), no. 4, 3919-3938. 

\bibitem{So}
C. D. Sogge, \emph{Oscillatory integrals and spherical harmonics}. Duke Math. J. \textbf{53}~(1986), no. 1, 43-65. 

\bibitem{Wo}W
T. H. Wolff, \emph{
A counterexample in a unique continuation problem}. 
Comm. Anal. Geom. \textbf{2}~(1994), no. 1, 79-102. 

\bibitem{Zhu1}
J. Zhu, \emph{Quantitative uniqueness for elliptic equations}.  Amer. J. Math. \textbf{138}~(2016), no. 3, 733-762.




\end{thebibliography}
\end{document}